\documentclass[preprint,12pt]{elsarticle}

\usepackage{amssymb}
\usepackage{amsmath}
\usepackage{color}

\numberwithin{equation}{section}
\newtheorem{theorem}{Theorem}[section]
\newtheorem{corollary}[theorem]{Corollary}

\newtheorem{lemma}[theorem]{Lemma}
\newtheorem{proposition}[theorem]{Proposition}
\newtheorem{remark}[theorem]{Remark}
\newtheorem{example}[theorem]{Example}

\newtheorem{definition}[theorem]{Definition}

\newproof{proof}{Proof}

\begin{document}

\begin{frontmatter}

\title{A Note on $*$-Clean Rings}

\author[1]{Jian Cui}
\ead{cui368@mail.ahnu.edu.cn}
\author[2]{Zhou Wang}
\ead{zhouwang@seu.edu.cn}
\address[1]{Department of Mathematics, Anhui Normal
University, Wuhu 241000, China}
\address[2]{Department of Mathematics,
Southeast University, Nanjing 210096, China}

\begin{abstract}
 A $*$-ring $R$
is called (strongly) $*$-clean if every element of $R$ is the sum of
a projection and a unit (which commute with each other). In this
note, some properties of $*$-clean rings are considered. In
particular, a new class of $*$-clean rings which called strongly
$\pi$-$*$-regular are introduced. It is shown that $R$ is strongly
$\pi$-$*$-regular if and only if $R$ is $\pi$-regular and every
idempotent of $R$ is a projection if and only if $R/J(R)$ is
strongly regular with $J(R)$ nil, and every idempotent of $R/J(R)$
is lifted to a central projection of $R.$ In addition, the stable
range conditions of $*$-clean rings are discussed, and equivalent
conditions among $*$-rings related to $*$-cleanness are obtained.
\end{abstract}


\end{frontmatter}

%% \linenumbers

%% main text
\section { \bf Introduction}

Rings in which every element is the product of a unit and an
idempotent are said to be \emph{unit regular}. Recall that an
element of a ring $R$ is \emph{clean} if it is the sum of an
idempotent and a unit, and $R$ is \emph{clean} if every element of
$R$ is clean (see \cite{Nicholson77}). Clean rings were introduced
by Nicholson in relation to exchange rings and have been extensively
studied since then. Recently, Wang et al. \cite{Wang11} showed that
unit regular rings have idempotent stable range one (i.e., whenever
$aR+bR=R$ with $a,b\in R$, there exists $e^2=e\in R$ such that $a+be
\in U(R)$, written ${\rm isr}(R)=1$ for short), and rings with ${\rm
isr(R)}=1$ are clean. In 1999, Nicholson \cite{Nicholson99} called
an element of a ring $R$ \emph{strongly clean} if it is the sum of a
unit and an idempotent that commute with each other, and $R$ is
\emph{strongly clean} if each of its elements is strongly clean.
Clearly, a strongly clean ring is clean, and the converse holds for
an abelian ring (that is, all idempotents in the ring are central).
Local rings and strongly $\pi$-regular rings are well-known examples
of strongly clean rings.

A ring $R$ is a \emph{$*$-ring} (or \emph{ring with involution}) if
there exists
an operation $*:R\rightarrow R$ such that for all $x,~y\in R$\\
$ \indent\indent\indent\indent\indent\indent(x+y)^*=x^*+y^*, \ \
(xy)^*=y^*x^*, \ \ ~{\rm and} ~(x^*)^*=x. $\\ An element $p$ of a
$*$-ring is a \emph{projection} if $p^2=p=p^*$. Obviously, $0$ and
$1$ are projections of any $*$-ring. A $*$-ring $R$ is
\emph{$*$-regular} \cite{Ber72} if for every $x$ in $R$ there exists
a projection $p$ such that $xR=pR.$  Following Va$\rm{\check{s}}$
\cite{Va}, an element of a $*$-ring $R$ is (\emph{strongly})
\emph{$*$-clean} if it can be expressed as the sum of a unit and a
projection (that commute), and $R$ is \emph{$($strongly$)$
$*$-clean} if all of its elements are (strongly) $*$-clean. Clearly,
$*$-clean rings are clean and strongly $*$-clean rings are strongly
clean. It was shown in \cite{ChenCui,LiZ} that there exists a clean
$*$-ring but not $*$-clean, and unit regular $*$-regular rings
(which called \emph{$*$-unit regular} rings in \cite{ChenCui}) need
not be strongly $*$-clean, which answered two questions raised by
Va$\rm{\check{s}}$ in \cite{Va}.

In this note, we continue the study of (strongly) $*$-clean rings.
In Section $2$, several basic properties of (strongly) $*$-clean
rings are investigated. Motivated by the close relationship between
strong $\pi$-regularity and strong cleanness, we introduce the
concept of strongly $\pi$-$*$-regular rings in Section $3$. The
structure of strongly $\pi$-$*$-regular rings is considered and some
properties of extensions are discussed. As we know, it is still an
open question that whether a strongly clean ring has idempotent
stable range one, or even has stable range one (see
\cite{Nicholson99}). In Section $4$, we extend ${\rm isr}(R)=1$ to
the $*$-version. We call a $*$-ring $R$ have \emph{projection stable
range one} (written ${\rm psr}(R)=1$) if, for any $a,b\in R,$
$aR+bR=R$ implies that $a+bp$ is a unit of $R$ for some projection
$p\in R$. It is shown that if $R$ is strongly $*$-clean then ${\rm
psr}(R)=1$, and if ${\rm psr}(R)=1$ then $R$ is $*$-clean.
Furthermore, several equivalent conditions among (strongly) clean
rings, (strongly) $*$-clean rings and $*$-rings with projection
(idempotent) stable range one are obtained.

Throughout this paper, rings are associative with unity. Let $R$ be
a ring. The set of all idempotents, all nilpotents and all units of
$R$ are denoted by $Id(R)$, $R^{\rm nil}$ and $U(R)$, respectively.
For $a\in R,$ the commutant of $a$ is denoted by ${\rm
comm}(a)=\{x\in R:ax=xa\}$. We write $M_n(R)$ for the ring of all
$n\times n$ matrices over $R$ whose identity element we write as
$I_n$. Let $\mathbb{Z}_n$ be the ring of integers modulo $n.$ For a
$*$-ring $R$, the symbol $P(R)$ stands for the set of all
projections of $R$.
\smallskip
%%%%%%%%%%%%%%%%%%%%%%%%%%%%%%%%%%%%%%%%%%%%%%%%%%%%%%%%%%%%%%%%%%%%%
%%%%%%%%%%%%%%%%%%%%%    Section 2   %%%%%%%%%%%%%%%%%%%%%%%%%%%%
%%%%%%%%%%%%%%%%%%%%%%%%%%%%%%%%%%%%%%%%%%%%%%%%%%%%%%%%%%%%%%%%%%%%%

\section { \bf $*$-Clean Rings}
In this section, some basic properties of $*$-clean rings are
discussed, and several examples related to $*$-cleanness are given.

\begin{example}\label{1}
$(1)$ Units, elements in $J(R)$ and nilpotents of a $*$-ring $R$ are
$*$-clean.\\
$(2)$ Idempotents of a $*$-regular rings are $*$-clean.
\end{example}

\begin{proof}
$(1)$ It is obvious.

$(2)$ Let $R$ be $*$-regular and $e\in Id(R)$. Then there exists a
projection $p$ such that $(1-e)R=pR.$ So we have $1-e=p(1-e)$ and
$p=(1-e)p,$ and hence $ep=0.$ Note that
$(e-p)(e-p)=e-ep-pe+p=e+p(1-e)=e+(1-e)=1.$ So $e-p \in U(R)$, and
$e=p+(e-p)$ is $*$-clean in $R.$
\end{proof}

By Example \ref{1}, every local ring with involution $*$ is
$*$-clean. In \cite{Va}, Va${\rm \check{s}}$ asked whether there is
an example of a $*$-ring that is clean but not $*$-clean. It was
answered affirmatively in \cite{ChenCui} and \cite{LiZ}. In fact,
one can construct some counterexamples based on the following.

\begin{example}\label{2014-1}
Let $R$ be a boolean $*$-ring. Then $R$ is $*$-clean if and only if
$*=1_R$ is the identity map of $R.$ In particular,
$R=\mathbb{Z}_2\oplus \mathbb{Z}_2$ with $(a,b)^*=(b,a)$ is clean
but not $*$-clean.
\end{example}

\begin{proof}
Note that every boolean ring is clean. Suppose that $R$ is
$*$-clean. Given any $a\in R.$ Then $-a=p+u=p+1=p-1$ for some $p\in
P(R).$ So we have $a=1-p \in P(R).$ Thus, $a^*=a$, which implies
$*=1_R$. Conversely, if $*=1_R,$ then every idempotent of $R$ is a
projection. Thus, $R$ is $*$-clean.
\end{proof}

\begin{lemma}\label{2014-2}
Let $R$ be a $*$-ring. If $2\in U(R)$, then for any $u^2=1,~u^*=u\in
R$ if and only if every idempotent of $R$ is a projection.
\end{lemma}

\begin{proof}
$(\Rightarrow).$ Let $e\in Id(R).$ Then $(1-2e)^2=1.$ So we have
$2e=2e^*,$ and thus $2(e-e^*)=0.$ Since $2\in U(R),$ $e=e^*.$ As
desired.

$(\Leftarrow).$ Given $u\in R$ with $u^2=1.$ Then $\frac{u+1}{2}\in
Id(R)$ since $(\frac{u+1}{2})^2=\frac{u^2+2u+1}{4}=\frac{u+1}{2}$.
Since every idempotent of $R$ is a projection, it follows from
$(\frac{u+1}{2})^*=\frac{u+1}{2}$ that $u^*=u.$
\end{proof}

The $*$-ring $R=\mathbb{Z}_2\oplus \mathbb{Z}_2$ in Example
\ref{2014-1} reveals that ``$2\in U(R)$" in Lemma \ref{2014-2}
cannot be removed.

\begin{corollary}\label{cui11}
Let $R$ be a $*$-ring with $2\in U(R)$. The following are equivalent$:$\\
$(1)$ $R$ is clean and every unit of $R$ is self-adjoint $($i.e., $u^*=u$ for every $u\in U(R))$.\\
$(2)$ $R$ is $*$-clean and $*=1_R$.
\end{corollary}

\begin{proof}
$(2)\Rightarrow(1)$ is trivial.

$(1)\Rightarrow(2)$. Let $a\in R.$ Then $a=e+u$ for some $e\in
Id(R)$ and $u \in U(R)$. Note that $(1-2e)^2=1.$  By Lemma
\ref{2014-2}, $e^*=e.$ Thus $a\in R$ is $*$-clean and $a^*=a$, and
so $*=1_R$.
\end{proof}

Recall that an element $t$ of a $*$-ring $R$ is \emph{self-adjoint
square root of $1$} if $t^2=1$ and $t^*=t$.

\begin{theorem}\label{9}
Let $R$ be a $*$-ring, the following are equivalent$:$\\
$(1)$ $R$ is $*$-clean and $2\in U(R)$.\\
$(2)$ Every element of $R$ is a sum of a unit and a self-adjoint
square root of $1$.
\end{theorem}

\begin{proof}
$(1)\Rightarrow(2)$. Let $a\in R.$ Then $\frac{1+a}{2}=p+u$ for some
$p\in P(R)$ and $u\in U(R).$ It follows that $a=(2p-1)+2u$ where
$(2p-1)^*=2p-1,~(2p-1)^2=1$ and $2u\in U(R).$

$(2)\Rightarrow(1)$. We first show that $2\in U(R)$. By hypothesis,
$1=x+v$ with $x^2=1$ and $v\in U(R).$ So we have $(1-v)^2=x^2=1,$
which implies that $v^2=2v.$ Since $v$ is a unit, $v=2 \in U(R).$
Given any $a\in R,$ then there exist $y,w\in R$ satisfying
$2a-1=y+w$ with $y^*=y,~y^2=1$ and $w\in U(R)$. Thus,
$a=\frac{y+1}{2}+\frac{w}{2}$ is a $*$-clean expression since
$(\frac{y+1}{2})^*=\frac{y+1}{2},(\frac{y+1}{2})^2=\frac{y+1}{2}$
and $\frac{w}{2} \in U(R)$.
\end{proof}

Camillo and Yu \cite{Camillo} showed that if $R$ is a ring in which
$2$ is a unit, then $R$ is clean if and only if every element of $R$
is the sum of a unit and a square root of $1.$ Indeed, by the proof
of Theorem \ref{9}, the condition $2\in U(R)$ is also necessary.

\begin{proposition}\label{11}
The following are equivalent for a $*$-ring $R:$\\
$(1)$ $R$ is $*$-clean and $0, 1$ are the only projections.\\
$(2)$ $R$ is clean ring and $0, 1$ are the only idempotents.\\
$(3)$ $R$ is a local ring.
\end{proposition}

\begin{proof}
$(2)\Rightarrow(3)$ follows from \cite[Lemma 14]{Nich04} and
$(3)\Rightarrow(1)$ follows by Example \ref{1}.

$(1)\Rightarrow(2).$ It suffices to show that the only idempotents
in $R$ are $0$ and $1$. For $e^2=e\in R,$ the hypothesis implies
that $e=p+u$ where $p\in P(R)=\{0,1\}$ and $u\in U(R).$ If $p=0$
then $e=u$ is a unit, so $e=1.$ If $p=1$ then $1-e=-u\in U(R),$ and
hence $e=0.$ As required.
\end{proof}

Let $I$ be an ideal of a $*$-ring $R$. We call $I$ is
\emph{$*$-invariant} if $I^*\subseteq I$. In this case, the
involution $*$ of $R$ can be extended to the factor ring $R/I$,
which is still denoted by $*$.

\begin{lemma}\label{21}
Let $R$ be $*$-clean. If $I$ is a $*$-invariant ideal of $R$, then
$R/I$ is $*$-clean. In particular, $R/J(R)$ is $*$-clean.
\end{lemma}

\begin{proof}
Since the homomorphism image of a projection (resp., unit) is also a
projection (resp., unit), the result follows.

Next we only need to prove that $J(R)$ is $*$-invariant. For any
$a^* \in (J(R))^*$, we show that $a^*\in J(R).$ Note that $a\in
J(R)$. Take any $x\in R.$ Then $1-x^*a\in U(R)$. Thus
$1-a^*x=(1-x^*a)^* $ is a unit of $R,$ as desired.
\end{proof}

Let $R$ be a $*$-ring. Then $*$ induces an involution of the power
series ring $R[[x]]$, denoted by $*$, where $(\sum_{i=0}^\infty
a_ix^i)^*=\sum_{i=0}^\infty a_i^*x^i$.

\begin{proposition}\label{22}
Let $R$ be a $*$-ring. Then $R[[x]]$ is $*$-clean if and only if $R$
is $*$-clean.
\end{proposition}

\begin{proof}
Suppose that $R[[x]]$ is $*$-clean. Note that $R\cong R[[x]]/(x)$
and $(x)$ is a $*$-invariant ideal of $R[[x]]$. By Lemma \ref{21},
$R$ is $*$-clean. Conversely, assume that $R$ is $*$-clean. Let
$f(x)=\sum_{i=0}^\infty a_ix^i \in R[[x]].$ Write $a_0=p+u$ with
$p\in P(R)$ and $u\in U(R)$. Then $f(x)=p+(u+\sum_{i=1}^\infty
a_ix^i)$, where $p\in P(R)\subseteq P(R[[x]])$ and
$u+\sum_{i=1}^\infty a_ix^i\in U(R[[x]])$. Hence $f(x)$ is $*$-clean
in $R[[x]]$.
\end{proof}

According to \cite[Proposition 13]{Nich04}, the polynomial ring
$R[x]$ is never clean. Hence, $R[x]$ is not $*$-clean for any
involution $*$.

\medskip
%%%%%%%%%%%%%%%%%%%%%%%%%%%%%%%%%%%%%%%%%%%%%%%%%%%%%%%%%%%%%%%%%%%%%
%%%%%%%%%%%%%%%%%%%%%%%%%%%%%%%%%%%%%%%%%%%%%%%%%%%%%%%%%%%
\section { \bf Strongly $\pi$-$*$-Regular Rings}

Strong $\pi$-regularity is closely related to strong cleanness. In
this section, we introduce the notion of strongly $\pi$-$*$-regular
rings which can be viewed as $*$-versions of strongly $\pi$-regular
rings. The structure and properties of strongly $\pi$-$*$-regular
rings are given.

\begin{lemma}\cite[Lemma 2.1]{LiZ}\label{20140}
Let $R$ be a $*$-ring. If every idempotent of $R$ is a projection,
then $R$ is abelian.
\end{lemma}

Due to \cite{ChenCui}, an element $a$ of a $*$-ring $R$ is
\emph{strongly $*$-regular} if $a=pu=up$ with $p\in P(R)$ and $u\in
U(R);$ $R$ is \emph{strongly $*$-regular} if each of its elements is
strongly $*$-regular. By \cite[Proposition 2.8]{ChenCui}, any
strongly $*$-regular element is strongly $*$-clean.

\begin{theorem}\label{20141}
Let $R$ be a $*$-ring. Then the following are equivalent for $a\in
R:$\\
$(1)$ There exist $e\in P(R),~u\in U(R)$ and an integer $m\geq 1$
such that $a^m=eu$ and $a,e,u$ commute with each other.\\
$(2)$ There exist $f\in P(R),~v\in U(R)$ such that $a=f+v,$ $fv=vf$
and $af \in R^{\rm nil}.$\\
$(3)$ There exists $p\in P(R)$ such that $p\in {\rm comm}(a)$,
$ap\in U(pRp)$ and $a(1-p)\in R^{\rm nil}.$\\
$(4)$ There exists $b\in {\rm comm}(a)$ such that $(ab)^*=ab,~b=bab$
and $a-a^2b \in R^{\rm nil}.$
\end{theorem}

\begin{proof}
$(1)\Rightarrow(2).$ Write $f=1-e.$ Clearly, $f\in P(R)$ and
$a^m-f\in U(R)$ with the inverse $u^{-1}e-f.$ From $af=fa,$ we have
$a-f:=v$ is a unit of $R$ (since
$(a-f)(a^{m-1}+a^{m-2}f+\cdots+af+f)=a^m-f \in U(R)$) and $fv=vf$.
It is clear that $(af)^m=a^mf=0.$

$(2)\Rightarrow(3).$ Set $p=1-f$. Then $p\in P(R),$ $ap=pa=vp\in
U(pRp)$ and $a(1-p)=af \in R^{\rm nil}.$

$(3)\Rightarrow(4)$. By $(3)$, $aw=wa=p$ for some $w\in U(pRp)$. So
we obtain $[a-(1-p)][w-(1-p)]=1-a(1-p)\in U(R)$ since $a(1-p)$ is
nilpotent, which implies that $a-(1-p)\in U(R).$ Let
$b=[a-(1-p)]^{-1}p.$ Then $b\in {\rm comm}(a),$ $bp=b$ and
$ab=[a-(1-p)]b=p\in P(R)$. Thus $(ab)^*=ab,~b=bp=bab$ and
$a-a^2b=a(1-ab)=a(1-p) \in R^{\rm nil}$.

$(4)\Rightarrow(1)$. Let $e=ab$. Then $(ab)^*=ab$ implies $e^*=e,$
and $bab=b$ yields $e^2=e$. So $e\in P(R).$ As $a-a^2b\in R^{\rm
nil},$ $a^m=a^me$ for some integer $m\geq 1.$ Take $u=a^m+(1-e)$ and
$u^{\prime}=b^me+(1-e).$ Then $uu^{\prime}=u^{\prime}u=1$. Hence,
$u\in U(R)$ and $a^m=a^me=ue$ with $a,e,u$ commuting with each
other.
\end{proof}

Recall that an element $a$ of a ring $R$ is \emph{strongly
$\pi$-regular} if $a^n \in a^{n+1}R \cap Ra^{n+1}$ for some $n \geq
1$ (equivalently, $a^n=eu$ with $e\in Id(R)$, $u\in U(R)$ and
$a,e,u$ all commute \cite{Nicholson99}); $R$ is \emph{strongly
$\pi$-regular} if every element of $R$ is strongly $\pi$-regular.
Based on the above, we introduce the following concept.

\begin{definition}
Let $R$ be a $*$-ring. An element $a\in R$ is called strongly
$\pi$-$*$-regular if it satisfies the conditions in Theorem
\ref{20141}; $R$ is called strongly $\pi$-$*$-regular if every
element of $R$ is strongly $\pi$-$*$-regular.
\end{definition}

\begin{corollary}\label{20142}
Any strongly $*$-regular element is strongly $\pi$-$*$-regular, and
any strongly $\pi$-$*$-regular element is strongly $*$-clean.
\end{corollary}

\begin{example}
$(1)$ Let $R=\mathbb{Z}_4$ and $*=1_R$.
Then $R$ is strongly $\pi$-$*$-regular. However, $2\in R$ is not strongly $*$-regular.\\
$(2)$ Let $R$ be a local domain with involution $*$ and $J(R)\neq
0$. Note that $P(R)=Id(R)=\{0,1\}$. So $R$ is strongly $*$-clean by
Proposition \ref{11}, but any power of a nonzero element in $J(R)$
can not expressed as the product of a projection and a unit.
\end{example}

Recall that a ring $R$ is \emph{$\pi$-regular} if for any  $a\in R$,
there exist $n\geq 1$ and $b\in R$ such that $a^n=a^nba^n.$ Strongly
$\pi$-regular rings and regular rings are $\pi$-regular (see
\cite{Nicholson99}). A ring $R$ is \emph{directly finite} if $ab=1$
implies $ba=1$ for all $a,b\in R.$ Abelian rings are directly
finite.

\begin{theorem}\label{20143}
The following are equivalent for a $*$-ring $R:$\\
$(1)$ $R$ is strongly $\pi$-$*$-regular.\\
$(2)$ $R$ is $\pi$-regular and every idempotent of $R$ is a
projection.\\
$(3)$ For any $a\in R,$ there exist $n\geq 1$ and $p\in P(R)$ such
that $a^nR=pR$, and $R$ is abelian.\\
$(4)$ For any $a\in R,$ there exist $n\geq 1$ such that
$a^n$ is strongly $*$-regular.\\
$(5)$ For any $a\in R$, there exist $p\in P(R)$ and $u\in U(R)$ such
that $a=p+u$, $ap\in R^{\rm nil}$; and $v^{-1}qv$ is a projection
for all $v\in U(R)$ and all $q\in P(R).$
\end{theorem}

\begin{proof}
$(1)\Rightarrow(2)$. Note that every strongly $\pi$-$*$-regular ring
is strongly $\pi$-regular and strongly $*$-clean. Thus $R$ is a
$\pi$-regular ring. By \cite[Theorem 2.2]{LiZ}, every idempotent of
$R$ is a projection.

$(2)\Rightarrow(3)$. For any $a\in R$, there exists $n\geq 1$ such
that $a^n=a^nxa^n$ for some $x\in R$. Write $a^nx=p$. Then $p\in
P(R)$ and $a^n=pa^n$. It is clear that $a^nR=pR.$ In view of Lemma
\ref{20140}, $R$ is abelian.

$(3)\Rightarrow(4)$. Let $e\in Id(R).$ Then $eR=pR$ for some $p\in
P(R).$ Since $R$ is abelian, we have $e=pe=ep=p.$ Thus, every
idempotent of $R$ is a projection. Given $a\in R,$ there exists
$n\geq 1$ and $q\in P(R)$ such that $a^nR=qR.$ So one gets
$a^n=qa^n$ and $q=a^nx$ for some $x\in R$, which implies
$a^n=a^nxa^n$. Next we show that $a^n-(1-q)$ is invertible. Note
that $[a^n-(1-q)][xq-(1-q)]=1.$ Then $a^n-(1-q):=u\in U(R)$ since
$R$ is directly finite. Multiplying the equation $a^n-(1-q)=u$ by
$p$ yields $a^n=a^nq=uq=qu,$ which implies that $a^n$ is strongly
$*$-regular.

$(4)\Rightarrow(5)$. For $e\in Id(R),$ $e=qv=vq$ for some $q\in
P(R)$ and $v\in U(R)$ by the assumption. Then $e=qv=e^2=qv^2$, and
so we obtain $q=qv=e,$ which implies that every idempotent of $R$ is
a projection. Clearly, $v^{-1}qv$ is a projection for all $v\in
U(R)$ and all $q\in P(R).$ Given $a\in R$ as in $(4),$
$a^n=(1-p)w=w(1-p)$ for some $p\in P(R)$ and $w\in U(R).$ Note that
$R$ is abelian. So we have $a^np=(ap)^n=0$ and
$(a-p)[a^{n-1}w^{-1}(1-p)-\sum_{i=0}^na^ip]=1,$ and hence $a-p \in
U(R)$ as $R$ is directly finite.

$(5)\Rightarrow(1)$. By $(5)$, every element of $R$ is $*$-clean. In
view of \cite[Theorem 2.2]{LiZ}, $R$ is abelian. Thus $R$ is
strongly $\pi$-$*$-regular by Theorem \ref{20141}(2).
\end{proof}

\begin{corollary}\label{20144}
Let $R$ be a $*$-ring. The following are equivalent$:$\\
$(1)$ $R$ is strongly $\pi$-$*$-regular.\\
$(2)$ $R/J(R)$ is strongly $\pi$-$*$-regular with $J(R)$ nil, every
projection of $R$ is central and every projection of $R/J(R)$ is
lifted to
a projection of $R$.\\
$(3)$ $R/J(R)$ is strongly $*$-regular with $J(R)$ nil, and every
idempotent of $R/J(R)$ is lifted to a central projection of $R.$
\end{corollary}

\begin{proof}
Write $\overline{R}=R/J(R)$. By Lemma \ref{21}, $\overline{R}$ is a
$*$-ring.

$(1)\Rightarrow(2)$. Clearly, $\overline{R}$ is strongly
$\pi$-$*$-regular. As $R$ is strongly $\pi$-regular, for any $a\in
J(R),$ there exist $m\geq 1,$ $e\in Id(R)$ and $u\in U(R)$ such that
$e=a^mu \in J(R).$ So $a^m=eu^{-1}=0,$ which implies that $J(R)$ is
nil. Note that $R$ is strongly $*$-clean. So the rest follows from
\cite[Corollary 2.11]{LiZ}.

$(2)\Rightarrow(3)$. By virtue of \cite[Corollar 2.11]{LiZ},
$\overline{R}$ is reduced (i.e., $\overline{R}^{\rm nil}=0$), and
every idempotent of $\overline{R}$ is lifted to a central projection
of $R.$ So we only need to prove that $\overline{R}$ is strongly
$*$-regular. Given any $x\in \overline{R}.$ By Theorem \ref{20141},
there exist $p\in P(\overline{R})$ and $v\in U(\overline{R})$ such
that $a=p+v$, $vp=pv$ and $ap\in \overline{R}^{\rm nil}=0.$ It
follows that $a=a(1-p)=v(1-p)=(1-p)v$ is strongly $*$-regular in
$\overline{R}.$

$(3)\Rightarrow(1)$. Since $\overline{R}$ is strongly regular, it is
reduced clean. By \cite[Corollary 2.11]{LiZ}, every idempotent of
$R$ is a projection. Note that $J(R)$ is nil and $\overline{R}$ is
$\pi$-regular. So $R$ is $\pi$-regular by \cite[Theorem 4]{Badawi}.
In view of Theorem \ref{20143}, $R$ is strongly $\pi$-$*$-regular.
\end{proof}

\begin{corollary}
Let $R$ be a $*$-ring. Then $R$ is strongly $*$-clean and
$\pi$-regular if and only if $R$ is strongly $\pi$-$*$-regular.
\end{corollary}

\begin{proof}
If $R$ is strongly $*$-clean and $\pi$-regular, by \cite[Theorem
2.2]{LiZ}, idempotents of $R$ are projections. So $R$ is strongly
$\pi$-$*$-regular by Theorem \ref{20143}. The other direction is
clear.
\end{proof}

For a $*$-ring $R$, the matrix ring $M_n(R)$ has a natural
involution inherited from $R:$ if $A=(a_{ij})\in M_n(R),$ $A^*$ is
the transpose of $(a_{ij}^*)$ (i.e., $A^*=(a_{ij}^*)^T=(a_{ji}^*)$).
Henceforth we consider $M_n(R)$ as a $*$-ring with respect to this
natural involution.

\begin{corollary}\label{2014-3}
Let $R$ be a $*$-ring. Then $M_n(R)$ is not strongly
$\pi$-$*$-regular for any $n\geq 2.$
\end{corollary}

Let $R$ be a $*$-ring and $S=pRp$ with $p\in P(R).$ Then the
restriction of $*$ on $S$ will be an involution of $S$, which is
also denoted by $*.$

\begin{corollary}
If $R$ is strongly $\pi$-$*$-regular, then so is $eRe$ for any $e\in
Id(R)$.
\end{corollary}

\begin{proof}
Let $S=eRe$ with $e\in Id(R).$ By hypothesis, $e$ is a projection of
$R$. So $S$ is a $*$-ring. It is well known that $S$ is strongly
$\pi$-regular (see also \cite[Lemma 39]{Bor08}). Clearly, every
idempotent of $S~(\subseteq R)$ is a projection. So the result
follows by Theorem \ref{20143}.
\end{proof}

Let $RG$ be the group ring of a group $G$ over a ring $R$. According
to \cite[Lemma 2.12]{LiZ}, the map $*:~RG\rightarrow RG$ given by
$(\sum_ga_g g)^*=\sum_g a^*_g g^{-1}$ is an involution of $RG$, and
is denoted by $*$ again.

\begin{corollary}
Let $R$ be a $*$-ring with artinian prime factors, $2\in J(R)$ and
$G$ be a locally finite $2$-group. Then $R$ is strongly
$\pi$-$*$-regular if and only if $RG$ is strongly $\pi$-$*$-regular.
\end{corollary}

\begin{proof}
Assume that $R$ is strongly $\pi$-$*$-regular. Then $Id(R)=P(R).$ In
particular, $R$ is abelian. So idempotents of $R$ coincide with
idempotents in $RG$ by \cite[Lemma 11]{CNZ}, and hence every
idempotent of $RG$ is a projection. Since $R$ is a ring with
artinian prime factors and $G$ is a locally finite $2$-group, $RG$
is a strongly $\pi$-regular ring by \cite[Theorem 3.3]{Chin}. In
view of Theorem \ref{20143}, $RG$ is strongly $\pi$-$*$-regular.

Conversely, $R$ is strongly $\pi$-regular by \cite[Proposition
3.4]{Chin}. Note that $Id(R)\subseteq Id(RG)$ and all idempotents of
$RG$ are projections. By Theorem \ref{20143}, $R$ is strongly
$\pi$-$*$-regular.
\end{proof}

Let $\mathbb{C}$ be the complex filed. It is well known that for any
$n\geq 1,$ the matrix ring $M_n(\mathbb{C})$ is strongly
$\pi$-regular. However, $M_n(\mathbb{C})$ is not strongly
$\pi$-$*$-regular whenever $n\geq 2$ by Corollary \ref{2014-3}. So
it is interesting to determine when a matrix of $M_n(\mathbb{C})$ is
strongly $\pi$-$*$-regular. The set of all $n\times 1$ matrices over
$\mathbb{C}$ is denoted by $\mathbb{C}^n.$

\begin{example}
Let $S=M_n(\mathbb{C})$ with $*$ the transpose operation. Then $A$
is strongly $\pi$-$*$-regular if and only if there exist
$e_1,e_2,\ldots,e_n \in \mathbb{C}^n$ such that $e_i^*e_j=0$ for
$i=1,\ldots,r$; $j=r+1,\ldots,n$, and $A=P\left(
\begin{smallmatrix}
C &0\\
0& N
\end{smallmatrix} \right)P^{-1}$ with $P=(e_1,e_2,\ldots,e_n)\in U(S),$ $C\in U(M_r(\mathbb{C}))$ and $N\in [M_{n-r}(\mathbb{C})]^{\rm
nil}$. In particular, any real symmetric matrix is strongly
$\pi$-$*$-regular.
\end{example}

\begin{proof}
Given $A\in S.$ Assume that $rank(A)=r.$ By the Jordan canonical
decomposition, there exists $P=(e_1,e_2,\ldots,e_n)\in U(S)$ such
that $A=P\left(
\begin{smallmatrix}
C &0\\
0& N
\end{smallmatrix} \right)P^{-1}$, where $e_i \in \mathbb{C}^n$ for all $i$, $C\in U(M_r(\mathbb{C}))$ and $N\in [M_{n-r}(\mathbb{C})]^{\rm
nil}$. Write $B=P\left(
\begin{smallmatrix}
C^{-1} &0\\
0& 0
\end{smallmatrix} \right)P^{-1}$. Then one easily gets that
$BA=AB,~B=BAB$ and $A-A^2B=P\left(
\begin{smallmatrix}
0 &0\\
0& N
\end{smallmatrix} \right)P^{-1}$ is nilpotent. Note that $B$ satisfies the above conditions is unique (see \cite{Bhaskara02}).
In view of Theorem \ref{20141}, $A$ is strongly $\pi$-$*$-regular if
and only if $(AB)^*=AB.$ Notice that $AB=P\left(
\begin{smallmatrix}
I_r &0\\
0& 0
\end{smallmatrix} \right)P^{-1}$ and
\begin{equation*}
\begin{split}
& ~~~\ \ ~\ \ (AB)^*=AB\\
&\Leftrightarrow (P^{-1})^*\left(
\begin{smallmatrix}
I_r &0\\
0& 0
\end{smallmatrix} \right)P^*=P\left(
\begin{smallmatrix}
I_r &0\\
0& 0
\end{smallmatrix} \right)P^{-1}\\
& \Leftrightarrow (P^*P)^{-1}\left(
\begin{smallmatrix}
I_r &0\\
0& 0
\end{smallmatrix} \right)(P^*P)=\left(
\begin{smallmatrix}
I_r &0\\
0& 0
\end{smallmatrix} \right)\\
& \Leftrightarrow \left(
\begin{smallmatrix}
I_r &0\\
0& 0
\end{smallmatrix} \right)(P^*P)=(P^*P)\left(
\begin{smallmatrix}
I_r &0\\
0& 0
\end{smallmatrix} \right)\\
& \Leftrightarrow P^*P=\left(
\begin{smallmatrix}
V_1 &0\\
0& V_2
\end{smallmatrix} \right) {\rm with}~V_1\in U(M_r(\mathbb{C}))~ {\rm
and}~V_2\in U(M_{n-r}(\mathbb{C}))\\
& \Leftrightarrow e_i^*e_j=0 {\rm~for ~all~}i\in
\{1,2,\ldots,r\},j\in \{r+1,r+2,\ldots,n\},
\end{split}
\end{equation*}
where $V_1=(e_1^*,e_2^*,\ldots,e_r^*)^T(e_1,e_2,\ldots,e_r)$;
$V_2=(e_{r+1}^*,e_{r+2}^*,\ldots,e_n^*)^T(e_{r+1},e_{r+2},\ldots,e_n).$

If $A\in S$ is a real symmetric matrix, then there exists an
orthogonal matrix $P$ (i.e., $P^{-1}=P^T=P^*$) such that $A=P\left(
\begin{smallmatrix}
I_r &0\\
0& 0
\end{smallmatrix} \right)P^{-1}.$ So the result follows.
\end{proof}

In view of \cite[Proposition 3]{Ber72}, the involution of a
$*$-regular ring $R$ is proper $($i.e., $x^*x=0$ implies that $x=0$
for all $x\in R$$)$.

\begin{remark}
If $R$ is strongly $\pi$-$*$-regular, then for any $x\in R$,
$x^*x=0$ implies $x\in R^{\rm nil}$.  Indeed, by Theorem
\ref{20141}, there exist $p\in P(R)$ and $u\in U(R)$ such that
$x^m=pu=up$ for some $m\geq 1$. Then
$0=(x^*)^mx^m=(x^m)^*x^m=u^*pu$, and thus $p=0,$ whence $x^m=0$.
\end{remark}

\medskip
%%%%%%%%%%%%%%%%%%%%%%%%%%%%%%%%%%%%%%%%%%%%%%%%%%%%%%%%%%%%%%%%%%%%%
%%%%%%%%%%%%%%%%%%%%%%%%%%%%%%%%%%%%%%%%%%%%%%%%%%%%%%%%%%%
\section { \bf Stable Range Conditions}

In \cite{Nicholson99}, Nicholson asked whether every strongly clean
ring has stable range one, and it is still open. Recall that a ring
$R$ is said to \emph{have idempotent stable range one} (written
isr$(R)$=1) provided that for any $a,~b\in R$, $aR+bR=R$ implies
that $a+be \in U(R)$ for some $e\in Id(R)$ (see
\cite{Chen99,Wang11}). If $e$ is an arbitrary element of $R$ (not
necessary an idempotent), then $R$ is said to \emph{have stable
range one.} Clearly, if isr$(R)=1$, then $R$ is clean and has stable
range one. We extend the notion of isr$(R)=1$ to $*$-versions.

\begin{definition}\label{41}
A $*$-ring $R$ is said to have projection stable range
one~$($written $\rm{psr} (R)=1$$)$ if for any $a,~b\in R$, $aR+bR=R$
implies there exists $p\in P(R)$ such that $a+bp \in U(R)$.
\end{definition}

The following result is motivated by \cite[Proposition 2]{Chen99}.
\begin{proposition}\label{42}
Let $R$ be a $*$-ring. The following are equivalent$:$\\
$(1)$ $\rm{psr}(R)=1$.\\
$(2)$ For any $a,~b\in R,$ $aR+bR=R$ implies there exists $p\in
P(R)$ such that $a+bp$ is right invertible.\\
$(3)$ For any $a,~b\in R$, $aR+bR=R$ implies there exists $p\in
P(R)$ such that $a+bp$ is left invertible.
\end{proposition}

\begin{proof} The proof is similar to that of \cite[Proposition
2]{Chen99}.

$(1)\Rightarrow(2)$ is clear.

$(2)\Rightarrow(3)$. Let $a,~b\in R$ with $aR+bR=R.$ Then there is a
projection $p\in R$ such that $a+bp=u$ is right invertible. Assume
that $uw=1$ for some $w\in R.$ Then $wR+(1-wu)R=R.$ So the
hypothesis implies there exists $q\in P(R)$ such that $w+(1-wu)q$ is
right invertible. Note that $u[w+(1-wu)q]=1$. Thus $w+(1-wu)q$ is
also left invertible, and hence invertible. This implies that $u \in
U(R)$.

$(3)\Rightarrow(1).$ Given any $a,b \in R$ with $aR+bR=R.$ Then
there exists $p\in P(R)$ such that $a+bp$ is left invertible. We may
let $v\in R$ with $v(a+bp)=1.$ Then $vR+0R=R.$ By hypothesis, we can
find a projection $q$ such that $v+0q=v$ is left invertible. So $v$
is a unit, which implies that $a+bp \in U(R).$ Therefore, ${\rm
psr}(R)=1.$
\end{proof}

For a $*$-ring $R$, it is clear that if ${\rm psr}(R)=1,$ then ${\rm
isr}(R)=1.$ However, there exists a $*$-ring with ${\rm isr}(R)=1$
but not satisfies ${\rm psr}(R)=1$.

\begin{example}\label{44}
Define the involution of $\mathbb{Z}_2$ by $*:x\mapsto x$. Let
$S=M_2(\mathbb{Z}_2)$. Then $S$ is a $*$-ring. In view of
\cite[Corollary 3.4]{Wang11}, ${\rm isr}(S)=1$ since $S$ is unit
regular. Notice that $P(S)=\{O,I_2,\left(
\begin{smallmatrix}
1 &0 \\
0& 0
\end{smallmatrix} \right),\left(
\begin{smallmatrix}
0 &0 \\
0& 1
\end{smallmatrix} \right)\}$,
and $\left(
\begin{smallmatrix}
1 &0 \\
0& 0
\end{smallmatrix} \right)S+\left(
\begin{smallmatrix}
0 &0 \\
1& 0
\end{smallmatrix} \right)S=S.$ However, $\left(
\begin{smallmatrix}
1 &0 \\
0& 0
\end{smallmatrix} \right)+\left(
\begin{smallmatrix}
0 &0 \\
1& 0
\end{smallmatrix} \right)P$ is not invertible for any $P\in P(S).$
Hence, ${\rm psr}(S)\neq 1.$
\end{example}

From Example \ref{44}, one can also find that the projection stable
range one property cannot be inherited to the matrix ring.

\begin{proposition}\label{43}
Let $R$ be a $*$-ring. If $\rm{psr} (R)=1$, then $R$ is $*$-clean.
\end{proposition}

\begin{proof}
For any $a\in R,$ the equation $aR+(-1)R=R$ implies that
$a+(-1)p=u\in U(R)$ for some $p\in P(R).$ So $a=p+u,$ and hence $R$
is $*$-clean.
\end{proof}

According to \cite[Proposition 4]{Va}, the ring in Example \ref{44}
is $*$-clean. So we conclude that the converse of Proposition
\ref{43} is not true.

Following Nicholson \cite{Nicholson77}, a ring $R$ is
\emph{exchange} if for every $a\in R,$ there exists $e^2=e \in aR$
such that $1-e\in (1-a)R.$ Clean rings are exchange, the converse
holds whenever the rings are abelian.  A $*$-ring $R$ is called
\emph{$*$-abelian} if every projection of $R$ is central \cite{Va}.

\begin{theorem}\label{cui}
Let $R$ be a $*$-ring. The following are equivalent$:$\\
$(1)$ ${\rm psr}(R)=1$ and $R$ is $*$-abelian.\\
$(2)$ For any $a,~b\in R$, $aR+bR=R$ implies there exists a
projection $p\in {\rm comm}(a)$ such that $a+bp \in U(R)$.\\
$(3)$ ${\rm isr}(R)=1$ and every idempotent of $R$ is a projection.\\
$(4)$ $R$ is clean $($or exchange$)$ and every idempotent of $R$ is a projection.\\
$(5)$ $R$ is $*$-clean and $*$-abelian.\\
$(6)$ $R$ is strongly $*$-clean.\\
$(7)$ For every $a\in R$, there exists a projection $p\in aR$ such
that $1-p\in (1-a)R.$
\end{theorem}

\begin{proof}
$(1)\Rightarrow(2)$ and $(3)\Rightarrow(4)$ are clear;
$(4)\Rightarrow(5)\Rightarrow(6)\Rightarrow(7)$ follows from
\cite[Theorem 2.2]{LiZ}.

$(2)\Rightarrow(3)$. We only need to show that all idempotents are
projections. Let $e\in Id(R).$ Then $eR+(-1)R=R.$ So there exists
$p\in P(R)$ such that $ep=pe$ and $e-p\in U(R)$. Note that
$(e-p)(1-e-p)=(1-e-p)(e-p)=0.$ Thus, $e=1-p \in P(R).$ Therefore,
every idempotent of $R$ is a projection.

$(7)\Rightarrow(1)$. Let $e\in Id(R).$ Then there exists a
projection $p\in eR$ such that $1-p\in (1-e)R.$ So we obtain $p=ep$
and $1-p=(1-e)(1-p)$. It follows that $e=p,$ and thus $Id(R)=P(R).$
In view of Lemma \ref{20140}, $R$ is abelian. Note that $R$ is
exchange. Then by \cite[Theorem 12]{Chen99}, ${\rm isr}(R)=1,$ and
hence ${\rm psr}(R)=1.$
\end{proof}

It is still unknown that whether strongly clean rings have stable
range one (\cite{Nicholson99}). However, we have an affirmative
answer of their $*$-versions.

\begin{corollary}\label{38}
If $R$ is a strongly $*$-clean ring, then ${\rm psr}(R)=1.$
\end{corollary}

The following example will reveal that the converse of Corollary
\ref{38} does not hold.

\begin{example}
Let $S=M_2(\mathbb{Z}_3)$. The involution of $S$ is defined by
$A\rightarrow A^*$, where $A^*$ is the transpose of $A\in S$. Then
$S$ is not strongly $*$-clean by \cite[Theorem 2.3]{ChenCui}. Since
$S$ is unit regular, ${\rm isr}(S)=1$ by \cite[Corollary
3.4]{Wang11}. In view of \cite[Lemma 7]{Chen}, we have
\begin{center}$Id(S)=\{O,I_2,\left(
\begin{smallmatrix}
x &y\\
z& 1-x
\end{smallmatrix} \right)~with~yz=x-x^2\}$,
\end{center} and
\begin{center}
$P(S)=\{O,I_2,\left(
\begin{smallmatrix}
1 &0 \\
0& 0
\end{smallmatrix} \right),\left(
\begin{smallmatrix}
0 &0 \\
0& 1
\end{smallmatrix} \right),\left(
\begin{smallmatrix}
2 &1 \\
1& 2
\end{smallmatrix} \right),\left(
\begin{smallmatrix}
2 &2 \\
2& 2
\end{smallmatrix} \right)\}$.
\end{center}
We next prove that ${\rm psr}(S)=1.$ Assume on the contrary. Then
there exist $A=\left(
\begin{smallmatrix}
a &b \\
c&d
\end{smallmatrix} \right)$ and $A^{\prime}=\left(
\begin{smallmatrix}
a^{\prime} &b^{\prime} \\
c^{\prime}& d^{\prime}
\end{smallmatrix} \right)$ with $AS+A^{\prime}S=S$ but
$A+A^{\prime}P$ is not a unit for any $P\in P(S).$ That is,
\begin{center}
${\rm det}(A+A^{\prime}P)=0.$
\end{center}
This implies the following system of equations$:$
\begin{align*}
ad-bc&=0~~~\indent\indent\indent\indent(i), \ \ & ad^{\prime}-b^{\prime}c&=0~~~\indent\indent\indent\indent(ii), \\
a^{\prime}d-bc^{\prime}&=0 ~~~\indent\indent\indent\indent(iii),\ \
&
a^{\prime}d^{\prime}-b^{\prime}c^{\prime}&=0~~~\indent\indent\indent\indent(iv),\\
ac^{\prime}-a^{\prime}c&=bd^{\prime}-b^{\prime}d~~~\indent ~~~(v).
\end{align*}
On the other hand, as ${\rm isr}(S)=1$, there exists $E\in
Id(S)\setminus P(S)$ such that $A+A^{\prime}E\in U(S)$. Then $E$
must be of the form $\left(
\begin{smallmatrix}
x &y\\
z& 1-x
\end{smallmatrix} \right)$ where $yz=x-x^2$. By Eqs. $(i)-(iv)$, we
obtain
\begin{center}
${\rm
det}(A+A^{\prime}E)=(ac^{\prime}-a^{\prime}c)y-(bd^{\prime}-b^{\prime}d)z.$
\end{center}
Next we show that
$ac^{\prime}-a^{\prime}c=bd^{\prime}-b^{\prime}d=0.$\\
\textbf{Case 1.} $c\neq 0.$ Multiplying Eq. (v) by $c$ and by
substituting $b^{\prime}c=ad^{\prime}$, we have
$(ac^{\prime}-a^{\prime}c)c=bd^{\prime}c-b^{\prime}dc=(bc-ad)d^{\prime}=0$
by Eq. (i).
Thus, $ac^{\prime}-a^{\prime}c=bd^{\prime}-b^{\prime}d=0.$\\
\textbf{Case 2.} $d\neq 0.$ Multiplying Eq. (v) by $d$ and by
substituting $a^{\prime}d=bc^{\prime}$, we have
$(bd^{\prime}-b^{\prime}d)d=ac^{\prime}d-a^{\prime}cd=(ad-bc)c^{\prime}=0$
by Eq. (i).
So $ac^{\prime}-a^{\prime}c=bd^{\prime}-b^{\prime}d=0.$\\
\textbf{Case 3.} $c=d=0.$ From Eq. $(ii)$ and $(iii)$, we get
$ad^{\prime}=bc^{\prime}=0.$ If $b\neq 0$, then $c^{\prime}=0$, it
follows that $ac^{\prime}-a^{\prime}c=0.$ If $a\neq 0$, then
$d^{\prime}=0,$ and so $bd^{\prime}-b^{\prime}d=0.$ Thus
$ac^{\prime}-a^{\prime}c=bd^{\prime}-b^{\prime}d=0.$

Therefore, ${\rm
det}(A+A^{\prime}E)=(ac^{\prime}-a^{\prime}c)y-(bd^{\prime}-b^{\prime}d)z=0$
for any case, which contradicts $A+A^{\prime}E\in U(S)$. Hence,
${\rm psr}(R)=1.$
\end{example}

By Theorem \ref{cui}, we have the following result immediately.

\begin{corollary}
Let $R$ be a $*$-ring. If $Id(R)=P(R),$ then the following are
equivalent$:$\\
$(1)$ $R$ is $($strongly$)$ clean.\indent\indent\indent$(2)$ $R$ is exchange.\\
$(3)$ $R$ is $($strongly$)$ $*$-clean. \indent\indent$(4)$ ${\rm
isr}(R)=1.$\indent~~~\indent\indent $(5)$ ${\rm psr}(R)=1.$
\end{corollary}

\bigskip

\centerline {\bf ACKNOWLEDGMENTS}
\vskip 2mm

This work was supported  by the
NNSF of China (No. 11326062, 11201064).

\end{document}